\documentclass[12pt]{amsart}

\usepackage{latexsym, amssymb, amsmath,ulem,soul,esint}
\usepackage{tikz}

\usepackage{amsfonts, graphicx}
\usepackage{graphicx,color}
\newcommand{\be}{\begin{equation}}
\newcommand{\ee}{\end{equation}}
\newcommand{\beq}{\begin{eqnarray}}
\newcommand{\eeq}{\end{eqnarray}}

\newtheorem{thm}{Theorem}[section]

\newtheorem{lma}{Lemma}[section]
\newtheorem{prop}{Proposition}[section]

\usepackage{mathrsfs} 

\usepackage{wasysym,stmaryrd}

\theoremstyle{remark}
\newtheorem{rem}{Remark}[section]
\numberwithin{equation}{section}

\def\be{\begin{equation}}
\def\ee{\end{equation}}
\def\bee{\begin{equation*}}
\def\eee{\end{equation*}}

\def\lf{\left}
\def\ri{\right}

\def\K{K\"ahler }

\def\KR{K\"ahler-Ricci }
\def\Ric{\text{\rm Ric}}
\def\Rm{\text{\rm Rm}}

\def\p{\partial}

\def\yheat{\lf(\frac{\p}{\p t}-(n-1)\Delta_{g(t)}\ri)}

\def\e{\varepsilon}

\def\a{{\alpha}}
\def\b{{\beta}}

\begin{document}

\title[]
{Gap Theorem on locally conformally flat manifold}

\author{Ming Hsiao}
\address[Ming Hsiao]{Department of Mathematics
National Taiwan University
and
National Center for Theoretical Sciences,
Math Division,
Taipei 10617,
Taiwan
}
\email{r12221002@ntu.edu.tw}

\author{Man-Chun Lee}
\address[Man-Chun Lee]{Department of Mathematics, The Chinese University of Hong Kong, Shatin, N.T., Hong Kong
}
\email{mclee@math.cuhk.edu.hk}

\renewcommand{\subjclassname}{
  \textup{2020} Mathematics Subject Classification}
\subjclass[2020]{ Primary 53E99, 53C20, 53C21
}

\date{\today}

\begin{abstract}
In this work, we study a gap phenomenon in locally conformally flat Riemannian manifolds with non-negative Ricci curvature. We construct complete solutions to the Yamabe flow that exhibit instantaneous bounded curvature as they evolve. Using this, we demonstrate that if the curvature decays quickly enough in an integral sense, then the manifold must be flat. This partially generalizes the results of Chen-Zhu and Ma.
\end{abstract}

\keywords{locally conformally flat, gap Theorem}

\maketitle

\markboth{Ming Hsiao, Man-Chun Lee}{}

\section{Introduction}
 
Let $(M^n,g_0)$ be a complete Riemannian manifold. 
We say that $(M,g)$ is \textit{locally conformally flat} if, for any point $p\in M$, there exists a coordinate chart conformal to the open set in Euclidean space. This is a high-dimensional analog of isothermal coordinate and allows for elegant structural results with curvature lower bounds, for example see  \cite{SchoenYau1988}.

In the compact case, there have been many studies in the literature, for example, see \cite{Cheng2001,Noronha1993} and the reference therein. In this work, we are primarily interested in studying complete non-compact conformally flat manifolds with $\Ric\geq 0$. Indeed, their topology has been fully classified by Carron-Herzlich \cite{CarronHerzlich2006}, see also the earlier work of Zhu \cite{Zhu1994}. More precisely if we restrict our attention to the non-compact case, the result of Carron-Herzlich implies that $(M,g_0)$ can only be locally the cylinder $\mathbb{R}\times \mathbb{S}^{n-1}$, the flat $\mathbb{R}^n$ or globally conformal to $\mathbb{R}^n$ with non-flat metric.  The Euclidean space $\mathbb{R}^n$ is then the only space that admits non-flat complete conformally flat metrics with decaying non-negative Ricci curvature. Indeed, it was already known by the work of Chen-Zhu \cite{ChenZhu2002} that the metric must be flat if the  metric satisfies fast curvature decay assumptions at infinity, see also \cite{Ma2016}. Furthermore, Carron-Herzlich \cite{CarronHerzlich2006} gave an explicit example showing that the result of Chen-Zhu is optimal within some class of metrics. And we also refer readers to \cite{ESS, Drees,MokSiuYau,GreeneWu,GPZ} for some classical results on gap phenomenon.

In this work, we are interested in the gap phenomenon within the class of metrics without a-priori bounded curvature conditions. This is partially motivated by the optimal gap Theorem of Ni \cite{Ni} concerning complete non-compact \K manifold with holomorphic bisectional curvature $\mathrm{BK}\geq 0$. Specifically, it was shown by Ni that if a complete \K manifold $(M,g_0)$ satisfies $\mathrm{BK}\geq 0$ and 
\begin{equation}
k(x_0,r):= \fint_{B_{g_0}(x,r)}\mathcal{R}(g_0)\, d\mathrm{vol}_{g_0}=o(r^{-2})
\end{equation}
as $r\to+\infty$, for some $x_0\in M$, then $(M,g_0)$ is necessarily flat. Here $\mathcal{R}(g)$ denotes the scalar curvature of metric $g$. The gap Theorem of Ni does not rely on any bounded curvature assumption, see also the earlier work of Chen-Zhu \cite{ChenZhu2003} using \KR flow. By now, it is still unclear to what extent Ni's optimal gap Theorem can be extended to Riemannian case. In contrast, Chen-Zhu \cite{ChenZhu2002} showed that under bounded curvature assumption, if $k(x,r)=o(r^{-2})$ uniformly in $x\in M$, then locally conformally flat $(M,g)$ with $\Ric\geq 0$ must be flat. Motivated by these two results, we study locally conformally flat $(M,g)$ with $\Ric\geq 0$ under stronger asymptotic assumption on $k(\cdot, r)$ but without assuming bounded curvature.

The following is our main result in this work:
\begin{thm}\label{thm:main}
For all $n\geq 5$, there exists $\e_0(n)>0$ such that the following holds: Suppose $(M,g_0)$ is a complete non-compact manifold which is locally conformally flat and has $\Ric(g_0)\geq 0$. If 
\begin{equation}
\int^\infty_0 r \fint_{B_{g_0}(x,r)}\mathcal{R}(g_0)\, d\mathrm{vol}_{g_0}\,  dr\leq \e_0
\end{equation}
for all $x\in M$, then $(M,g_0)$ is flat.
\end{thm}

We stress that the assumption is scaling invariant. We remark here that $(M,g_0)$ is immediately globally conformally flat under the assumption in Theorem~\ref{thm:main} if it is non-flat, thanks to the work of Carron-Herzlich \cite{CarronHerzlich2006}. So by way of contradiction, we might assume $M$ to be topologically $\mathbb{R}^n$, but we do not need this fact.  This is the locally conformally flat analogy of the earlier work of Chan and the second named author \cite{ChanLee2023}, where locally conformally flat with $\Ric\geq 0$ is replaced by manifold with non-negative complex sectional curvature. 
\medskip

Our strategy is deeply influenced by the work of Hamilton \cite{HamiltonJDG} in Ricci flow where parabolic deformation is used to regularize the metric to the canonical model. In the case of locally conformally flat, we use the Yamabe flow. The Yamabe flow $g(t)$ is a one parameter family of metrics such that  
\begin{equation}
\partial_t g(t)=-\mathcal{R}_{g(t)}\cdot  g(t).
\end{equation}
In particular, this is a conformal deformation of metrics. The Yamabe flow was first proposed by Richard Hamilton in the 1980s as a tool for
constructing metrics of constant scalar curvature in a given conformal class \cite{Hamilton}. In the compact case, Hamilton proved that the Yamabe flow has a global solution for every initial metric. The asymptotic behavior of the Yamabe flow was first studied by Chow \cite{Chow1992} where he proved that the flow approaches a metric of constant scalar curvature provided that the initial metric is locally conformally flat and has positive Ricci curvature.

In the non-compact case, this circle of idea was extended by Chen-Zhu \cite{ChenZhu2002} in order to regularize the metric so that vanishing curvature at infinity can be recognized through Hamilton-Chow's harnack inequality \cite{Chow1992}.  The major difficulty is to construct global solution Yamabe flow in a way that maximum principle is available. In \cite{ChenZhu2002}, it is based on constructing bounded curvature solution for a short-time, and studying the finite time singularity for bounded curvature solution. 

To overcome the absence of bounded curvature, we take an alternative way which is more local in nature. We use some idea in \cite{Ma2016} to construct global solution using potential theory, based on the "linear" nature of Yamabe flow. Unlike bounded curvature solution, the major challenge is to apply maximum principle to obtain geometric control from time-zero geometry. Furthermore although the Yamabe flow is relatively more linear than the Ricci flow, it does not behave analytically well in comparison type argument. In particular, strongly contrast with the Ricci flow, the distance $d_{g(t)}$ of the Yamabe flow $g(t)$ does not behave well, to the best of the author acknowledge. We mange to overcome this under the strong (uniform) asymptotic condition of $k(\cdot,r)$.
This is based on constructing uniformly conformal Yamabe flow where a-priori estimate of dirichlet conjugate heat kernel can be made available using idea of Bamler, Cabezas-Riva and Wilking \cite{BamlerCabezasWilking2019}. This allow us to construct Yamabe flow with instantaneously bounded non-negative Ricci curvature. This reduces the question back to bounded curvature setting so that the "relatively" standard theory can be employed.

The paper is organized as follows. In section 2, we will derive and recall some evolutions of curvature along locally conformally flat Yamabe flow. In section 3, we discuss the heat kernel estimate which will be used in section 4 to localize the curvature estimate. In section 5, we use argument of Ma to construct solution and will show that the solution has quantitative curvature estimate for all t. As an application, we show that the initial data is flat.

\vskip0.2cm

{\it Acknowledgement}: The authors would like to thank Pak-Yeung Chan for insightful discussion. The first-named author was supported by the National Science and Technology Council under grant No. 112-2115-M-002-004-MY3 and a direct grant of National Center for Theoretical Sciences. The second-named author was partially supported by Hong Kong RGC grant (Early Career Scheme) of Hong Kong No. 24304222 and No. 14300623, a direct grant of CUHK and a NSFC grant No. 12222122.

\section{Evolution equations along Yamabe flows}

We intend to regularize the given metric $g_0$ on $M$ using parabolic method. To make better use of the locally conformally flat nature, it is more sensitive to consider conformal perturbation. In this regard, we use the Yamabe flow which is the following evolution equation for metric: 
\begin{equation}
\left\{
\begin{array}{ll}
\partial_t g(t)=-\mathcal{R}(g(t)) g(t);\\[2mm]
g(0)=g_0
\end{array}
\right.
\end{equation}
where $\mathcal{R}(g)$ denotes the scalar curvature of $g$. In this section, we will collect some evolution equation along  Yamabe flow which will be used in this article. 

Similar to the Ricci flow, the Yamabe flow is favorable to (some) curvature positivity. For instances, the scalar curvature tends to remain non-negative generally, which can be seen by the following evolution equation.
\begin{lma}\label{lma:evo-R}
Along a Yamabe flow, the scalar curvature $\mathcal{R}$ of $g(t)$ satisfies 
\begin{equation}
\yheat \mathcal{R}=\mathcal{R}^2.
\end{equation}
\end{lma}
\begin{proof}
This follows from \cite[Lemma 2.2]{Chow1992}. 
\end{proof}

When $g(t)$ is locally conformally flat, it was observed by Chow that the Ricci curvature also satisfies  a similar evolution equation. 
\begin{lma}\label{lma:evo-Ric}
Along a locally conformally flat Yamabe flow, the Ricci curvature $\Ric(g(t))$ of $g(t)$ satisfies 
\begin{equation}
\yheat R_{ij}=\frac1{n-2} B_{ij} 
\end{equation}
where 
\begin{equation}
B_{ij}=(n-1)|\Ric|^2 g_{ij}+n\mathcal{R} R_{ij} -n(n-1)R_{ik}R^k_j -\mathcal{R}^2 g_{ij}.
\end{equation}
In particular, if we denote $\lambda_1\leq...\leq \lambda_n$ to be the eigenvalues of $\Ric(g(t))$ with respect to $g(t)$, then 
\begin{equation}
B_{11}=\frac12 \sum_{j,k\neq 1} (\lambda_j-\lambda_k)^2 +(n-2) \sum_{j\neq 1} (\lambda_j-\lambda_1)\lambda_1.
\end{equation}
\end{lma}
\begin{proof}
This follows from \cite[Lemma 2.4]{Chow1992}. 
\end{proof}

When $M$ is compact, it follows directly from Lemma~\ref{lma:evo-Ric} and  Hamilton's tensor maximum principle that $\Ric\geq 0$ is preserved along the locally conformally flat Yamabe flow. In the non-compact case, the maximum principle is more involved and we need the following evolution equation for the Ricci curvature lower bound. For each $(x,t)\in M\times [0,T]$, we define
\begin{equation}
\ell(x,t):= \inf\{s>0: \Ric(x,t)+s g(x,t)\geq 0\}.
\end{equation}

\begin{lma}\label{lma:evo-Ric-new}
Along a locally conformally flat Yamabe flow, the function $\ell$ satisfies 
\begin{equation}
\yheat \ell \leq 2\mathcal{R}\ell+n \ell^2.
\end{equation}
in the sense of barrier.
\end{lma}
\begin{proof}
At $(x_0,t_0)$, we might assume the smallest eigenvalues of $\Ric(x_0,t_0)$ is negative as otherwise $\ell\equiv 0$ locally around $(x_0,t_0)$ so that the conclusion holds trivially. Thus, we might assume $\Ric +\ell g \in \mathrm{C}_{\Ric\geq 0}$. By compactness of the sphere bundle, we may find an orthonormal frame $\{e_i\}_{i=1}^n$ so that it diagonalize $\Ric(g(t))$ with respect to $g(t)$. If we denote its eigenvalues as $\lambda_1\leq ...\leq \lambda_n$, then $\Ric|_{(x_0,t_0)}(e_1,e_1)=-\ell(x_0,t_0)$. We use Uhlenbeck trick to simplify the computation as follow: we extend $\{e_i\}_{i=1}^n$ around $x_0$ at $t=t_0$ to $\{E_i(t_0)\}_{i=1}^n$ by parallel transport with respect to $g(t_0)$ so that $\nabla E_i=\Delta E_i=0$ at $(x_0,t_0)$. We then use the ODE: $\partial_t E_i=\frac12\mathcal{R}E_i$ to extend it to $t\approx t_0$ so that $g(E_i,E_j)=\delta_{ij}$ locally around $(x_0,t_0)$. By assumption, $\varphi(x_0,t_0)=\ell(x_0,t_0)$ and $\ell\leq \varphi$ locally near $(x_0,t_0)$. Moreover, by Lemma~\ref{lma:evo-Ric} and evolution of the frame, the function $\varphi(x,t):=-\Ric(E_1,E_1)$ satisfies 
\begin{equation}\label{eqn:bundle-Ric}
\begin{split}
&\quad \yheat \Big|_{(x_0,t_0)}\varphi\\
&=-\frac1{n-2} \left[ (n-1)|\Ric|^2 +2(n-1)\mathcal{R} \lambda_1 -n(n-1)\lambda_1^2 -\mathcal{R}^2\right].
\end{split}
\end{equation}

We now further analyze the quadratic term more carefully. For convenience, we denote $\sum_{i=2}^n \lambda_i=\mu$ so that 
\begin{equation}
\begin{split}
&\quad (n-1)|\Ric|^2 +2(n-1)\mathcal{R} \lambda_1 -n(n-1)\lambda_1^2 -\mathcal{R}^2\\
&=(n-1)\sum_{i=1}^n \lambda_i^2-\left(\sum_{i=1}^n \lambda_i \right)^2+2(n-1)\left(\sum_{i=1}^n \lambda_i \right) \lambda_1 -n(n-1)\lambda_1^2 \\
&\geq \mu^2 +(n-1)\lambda_1^2-(\mu^2+2\mu \lambda_1 +\lambda_1^2)+2(n-1)(\mu+\lambda_1)\lambda_1 -n(n-1)\lambda_1^2\\
&=2(n-2)\mu\lambda-(n-2)^2 \lambda_1^2\\
&=(n-2)\left(-n\varphi^2-2 \mathcal{R} \varphi \right)
\end{split}
\end{equation}
Putting this back to \eqref{eqn:bundle-Ric}, this completes the proof.
\end{proof}

\section{Heat kernel estimates}

In this section, we consider the  heat kernel along the Yamabe flow and discuss some estimates. Given a complete solution to the Yamabe flow $g(t),t\in [0,T]$ on a smooth manifold $M$. Let $\Omega$ be a open set with smooth boundary. We consider the following heat kernel with dirichlet boundary condition:
\begin{equation}
\left\{
\begin{array}{ll}
\displaystyle \left(\frac{\partial}{\partial t}-(n-1)\Delta_{g(t)}-\frac{n}2 \mathcal{R}_{g(x,t)}\right) K_{\Omega}(x,t;y,s)=0,\;\;\text{for}\; x\in \Omega;\\[4mm]
\displaystyle\lim_{t\to s^+}K_\Omega(x,t;y,s)=\delta_y(x),\;\;\text{for}\; x\in \Omega;\\[4mm]
K_\Omega(x,t;y,s)=0,\;\;\text{for}\; x\in\partial\Omega,
\end{array}
\right.
\end{equation}
for $y\in \Omega$ and $t\in (s,T]$. Its conjugate is given by 

\begin{equation}
\left\{
\begin{array}{ll}
\displaystyle \left(\frac{\partial}{\partial s}+(n-1)\Delta_{g(s)}\right) K_{\Omega}(x,t;y,s)=0,\;\;\text{for}\; y\in \Omega;\\[4mm]
\displaystyle\lim_{s\to t^-}K_\Omega(x,t;y,s)=\delta_x(y),\;\;\text{for}\; y\in \Omega;\\[4mm]
K_\Omega(x,t;y,s)=0,\;\;\text{for}\; y\in\partial\Omega,
\end{array}
\right.
\end{equation}
for $x\in \Omega$ and $s\in [0,t)$. Particularly for $0\leq s<t<T$, we have
\begin{equation}\label{eqn:heat-mono}
\int_\Omega K_\Omega(\cdot,t;y,s)\,d\mathrm{vol}_{g(t)}\leq 1.
\end{equation}
This will play crucial role in obtaining estimates.
\medskip

Following the idea in \cite{ChanLee2023}, we intend to use heat kernel to control the curvature along the Yamabe flow, whenever it exists. To achieve this, we need to control the heat kernel $K_\Omega$ above, under scaling invariant decay $|\Rm(g(t))|\leq \a t^{-1}$. Unlike the Ricci flow, it is unclear to the authors how the distance function $d_{g(t)}$ should behave in general, under solely $|\Rm(g(t))|\leq \a t^{-1}$. In the following, we observe an estimate of $K_\Omega$ under Yamabe flow which is in addition uniformly conformal. This will be sufficient for our application.

\begin{prop}\label{prop:heatkernelestimate}
    For any $\a>1$, there exists $C=C(n,\a)>0$ such that the following holds. Suppose  $(M,g(t))_{t\in[0,1]}$ is a smooth solution to the Yamabe flow so that 
    \begin{equation}
        |{\Rm}(g(t))|\leq\frac{\a}{t}\text{ and }\a^{-1}g(0)\leq g(t)\leq \a g(0),
    \end{equation}
for all $(x,t)\in M\times(0,T]$. If $x_0\in M$ is a point such that  $B_{g(t)}(x_0,4)\Subset M$  for all $t\in[0,T]$ and $\Omega\Subset M$ is a  pre-compact domain with a smooth boundary so that $\Omega\Subset B_{g(t)}(x_0,1)$ for all $t\in[0, T]$. Then the heat kernel $K_\Omega$ satisfies
    \begin{equation}
        K_{\Omega}(x,t;y,s)\leq\frac{C}{\mathrm{Vol}_{g_0}\left(B_{g_0}(x,\sqrt{t-s}) \right)}\cdot \exp\left(-\frac{d_{g_0}(x,y)^2}{C(t-s)} \right)
    \end{equation}
    for all $0\leq s< t\leq T$ and $x,y\in\Omega$.
\end{prop}

When $g(t)$ is a solution to the Ricci flow, the corresponding heat kernel estimate was established by  \cite{BamlerCabezasWilking2019} in the complete case, under $|\Rm(g(t))|\leq \a t^{-1}$ and $\mathrm{inj}(g(t))\geq \sqrt{\a^{-1}t}$ for some $\a>0$. It was later observed by the second named author and Tam \cite{LeeTam2022} that the same result also holds in local sense, see also \cite{ChanLee2023}. 
%
%
\begin{proof}[Proof of Proposition~\ref{prop:heatkernelestimate}]
The proof is almost identical to that of \cite[Proposition 2.1]{ChanLee2023} which in turn is based on the argument in \cite[Proposition 3.1]{BamlerCabezasWilking2019} using \eqref{eqn:heat-mono}. We use the metric equivalence instead to take care of the distance distortion in Ricci flow.
\end{proof}

\section{A-priori local estimates}

In this section, we will localize curvature estimates for locally conformally flat Yamabe flow with scaling invariant control. The following localized maximum principle is the core tool, which will be used repeatedly in this work.

\begin{lma}\label{lma:localMP}
Suppose $g_0$ is a metric on $M$ and $x_0\in M$ is a point such that 
\begin{enumerate}
\item $B_{g_0}(x_0,r)\Subset M$;
\item $\Ric(g_0)\geq -(n-1)r^{-2}$.
\end{enumerate}
Let $g(t)$ be a smooth solution to the Yamabe flow on $M\times [0,Tr^2]$ (not necessarily complete) with $g(0)=g_0$ such that 
\begin{enumerate}
\item[(i)] $|\Rm(g(t))|\leq \a t^{-1}$;
\item[(ii)] $\mathcal{R}(g(t))\geq -n(n-1)r^{-2}$;
\item[(iii)] $\a^{-1}g_0\leq g(t)\leq \a g_0$
\end{enumerate}
on $M\times (0,Tr^2]$ for some $\a>1$ and $r>0$. If $\ell(x,t)$ is a non-negative continuous function on $M\times [0,Tr^2]$ such that $\ell(0)\leq r^{-2}$, $\ell(x,t)\leq \a t^{-1}$ on $M\times (0,T]$ and satisfies 
\begin{equation}
\yheat \ell\leq \b\mathcal{R}\ell +K\ell^2
\end{equation}
in the distribution sense, for some $K\geq 0$ and $0\leq \b<\frac{n}2$.  Then there is $\Lambda(\a,\b,K,n)>0$ such that for all $t\in (0,Tr^2]$, $\ell(x_0,t)\leq \Lambda r^{-2}$.
\end{lma}
\begin{proof}We follow the argument in \cite{LeeTam2022}. By scaling, we assume $r=1$. Without loss of generality, we might assume $T\leq 1$ since $\ell\leq \a t^{-1}$. We will use $C_i$ to denote any constant depending only on $\a,\b,K,n$. For $x\in B_{g_0}(x_0,1)$, we let $$\rho(x):=\sup\{s\in (0,1): B_{g_0}(x,s)\Subset B_{g_0}(x_0,1) \}.$$ 

Let $\Lambda>0$ be a large constant to be chosen. If $\ell< \Lambda \rho^{-2}$ for all $(x,t)\in B_{g_0}(x_0,1)\times [0,T]$, then we are done. Otherwise, there is $t_1\in (0,T)$ such that $\ell< \Lambda \rho^{-2}$ for all $(x,t)\in B_{g_0}(x_0,1)\times [0,t_1]$, and for some $x_1\in B_{g_0}(x_0,1)$, we have $\ell(x_1,t_1)=\Lambda \rho^{-2}(x_1)$. Denote $r_1=\frac12 \rho(x_1)$. We consider $B_{g_0}(x_1,r_1)\times [0,t_1]$ where $\ell\leq \Lambda \rho^{-2}\leq \Lambda r_1^{-2}$. 

Consider the parabolic rescaling of metrics and $\ell$: $\tilde g(t)=r_1^{-2}g(r_1^2t)$ and $\tilde\ell(t)=r_1^2\ell(r_1^2t)$ on $B_{\tilde g_0}(x_1,1)\times [0,\tilde t_1]$ where $\tilde t_1=r_1^{-2}t_1$. We note that the assumptions are invariant under the parabolic rescaling. Moreover, in the tile picture, we have $\tilde\ell(0)\leq 1$, $\tilde\ell\leq \Lambda$ on $B_{\tilde g_0}(x_1,1)\times [0,\tilde t_1]$ and $\tilde \ell(x_1,\tilde t_1)=\frac14 \Lambda$. 

We construct cut-off function as follows. Fix a smooth non-increasing function $\phi:[0,+\infty)\to[0,1]$ such that $\phi\equiv 1$ on $[0,\frac14]$, vanishes outside $[0,1]$ and satisfies $|\phi'|^2\leq 10^3\phi,\phi''\geq -10^3\phi$. We write the Yamabe flow using its conformal form, $\tilde g(t)=u^\frac4{n-2} \tilde g_0$  and will use the function 
\begin{equation}\label{eqn:cut-ff}
\Phi(x,t)=\phi\left(d_{\tilde g_0}(x,x_1)\right)\cdot u^\gamma
\end{equation} 
where $\gamma=1+3(n-1)^{-1}$, so that 
\begin{equation}
\begin{split}
\left(\frac{\partial}{\partial t}-(n-1)\Delta_{\tilde g(t)} \right) \Phi&=-(n-1) u^\gamma\cdot \Delta_{\tilde g(t)} \phi  \\
&\quad +\phi  \cdot \left(\frac{\partial}{\partial t}-(n-1)\Delta_{\tilde g(t)} \right)  u^\gamma\\
&\quad -(n-1) \langle \nabla \phi,\nabla u^\gamma\rangle_{\tilde g(t)}\\
&=\mathbf{I}+\mathbf{II}+\mathbf{III}.
\end{split}
\end{equation}

We estimate each term one by one. We first handle $\mathbf{I}$. By conformal formula and choice of $\phi$,
\begin{equation}\label{eqn:cutoff}
\begin{split}
-\Delta_{\tilde g(t)}\phi&=-u^{-\frac{4}{n-2}}\phi'\left(\Delta_{\tilde g_0}d_{\tilde g_0}(x,x_{1})+2\nabla_{\tilde g_0}\log u\cdot\nabla_{\tilde g_0}d_{\tilde g_0}(x,x_{1})\right)\\
    &\quad -u^{-\frac{4}{n-2}}\phi''|\nabla_{\tilde g_0}d_{\tilde g_0}(x,x_{1})|^{2}\\
    &\leq C_0\left(1+|\phi|^{1/2}|\nabla \log u|_{\tilde g(t)}\right)
\end{split}
\end{equation}
where we have used Laplacian comparison on $\tilde g_0$. Therefore, 
\begin{equation}\label{eqn:lap-compar}
\mathbf{I}\leq C_1 \left(1+|\phi|^{1/2}|\nabla \log u|_{\tilde g(t)}\right).
\end{equation}
For $\mathbf{III}$, we might argue similarly (and simpler) that 
\begin{equation}
\mathbf{III}\leq C_2|\phi|^{1/2}|\nabla \log u|_{\tilde g(t)}
\end{equation}

Using Yamabe flow equation, $\mathcal{R}(\tilde g(t))\geq -n(n-1)r_1^2$ and our choice of $\gamma$, we also know 
\begin{equation}\label{eqn:evo-u}
\begin{split}
&\quad \left(\frac{\partial}{\partial t}-(n-1)\Delta_{\tilde g(t)} \right) u^{\gamma}\\
&=\gamma u^{\gamma-\frac{4}{n-2}}\left[(2-(n-1)(\gamma-1))|\nabla\log u|_{\tilde g_0}^{2}-\frac{n-2}{4}\mathcal{R}(\tilde g(t))\right]\\
&\leq - u^\gamma|\nabla\log u|_{\tilde g(t)}^{2}+C_3 ,
\end{split}
\end{equation}
and thus $\mathbf{II}\leq -\Phi |\nabla\log u|_{\tilde g(t)}^{2}+C_3$. To summarize, we obtain
\begin{equation}\label{eqn:ctoff-evo}
\begin{split}
\left(\frac{\partial}{\partial t}-(n-1)\Delta_{\tilde g(t)} \right) \Phi&\leq  C_4\left(1+|\phi|^{1/2}|\nabla \log u|_{\tilde g(t)}\right)-\Phi  |\nabla\log u|_{\tilde g(t)}^{2}+C_3\\
&\leq \frac{n}2\tilde{\mathcal{R}} \Phi-\frac12 \Phi |\nabla\log u|_{\tilde g(t)}^{2}+C_5
\end{split}
\end{equation}
where we have used $\tilde{\mathcal{R}}\geq -n(n-1)r_1^2$.

\medskip
On the other hand, the function $F=e^{-K\Lambda t}\tilde\ell$ satisfies 
\begin{equation}\label{eqn:eqn=F}
\left(\frac{\partial}{\partial t}-(n-1)\Delta_{\tilde g(t)} \right) F^{1+\e}\leq (1+\e)\b \tilde{\mathcal{R}} F^{1+\e}-\e (n-1) (1+\e) F^{\e-1}|\nabla F|^2
\end{equation}
where we will choose $\e=(2\b)^{-1} n-1>0$. Hence, the function 
\begin{equation*}
\Psi(x,t)=\phi\left(d_{\tilde g_0}(x,x_1)\right)\cdot  F^{1+\e}(x,t) 
\end{equation*}
satisfies 
\begin{equation}
\begin{split}
\left(\frac{\partial}{\partial t}-(n-1)\Delta_{\tilde g(t)} \right) \Psi&=-(n-1) F^{1+\e}\cdot \Delta_{\tilde g(t)} \phi  \\
&\quad +\phi  \cdot \left(\frac{\partial}{\partial t}-(n-1)\Delta_{\tilde g(t)} \right)  F^{1+\e}\\
&\quad -(n-1)(1+\e)F^\e  \langle \nabla \phi,\nabla F \rangle_{\tilde g(t)}\\
&=\mathbf{A}+\mathbf{B}+\mathbf{C}
\end{split}
\end{equation}
where \eqref{eqn:eqn=F} implies 
\begin{equation}
\begin{split}
\mathbf{B}&\leq  (1+\e)\b \tilde{\mathcal{R}} \Psi-\e (n-1) (1+\e) \phi F^{\e-1}|\nabla F|_{\tilde g(t)}^2.
\end{split}
\end{equation}

Using \eqref{eqn:cutoff}, we can also obtain
\begin{equation}
\begin{split}
\mathbf{A}\leq C_6 F^{1+\e} \left(1+|\phi|^{1/2} |\nabla \log u|_{\tilde g(t)}\right).
\end{split}
\end{equation} 

Finally, using $|\nabla d_{\tilde g_0}|_{\tilde g_0}=1$ we have 
\begin{equation}
\mathbf{C}\leq C_7 F^\e |\phi|^{1/2} |\nabla F|_{\tilde g(t)}
\end{equation}
and therefore,
\begin{equation}
\begin{split}
\left(\frac{\partial}{\partial t}-(n-1)\Delta_{\tilde g(t)} \right) \Psi
&\leq \frac{n}2  \tilde{\mathcal{R}} \Psi+C_8 F^{1+\e}+C_6 F^{1+\e} |\phi|^{1/2}|\nabla \log u|_{\tilde g(t)}.
\end{split}
\end{equation}

By combining it with the evolution inequality concerning $\Phi$, the function $G:=\Phi+L\Psi$ for some large $L>1$, satisfies
\begin{equation}
\begin{split}
 \left(\frac{\partial}{\partial t}-(n-1)\Delta_{\tilde g(t)} \right) G
&\leq \frac{n}2 \tilde{\mathcal{R}} G+C_5 L-\frac{L}{2} \phi u^\gamma |\nabla\log u|_{\tilde g(t)}^2\\
&\quad +C_8 F^{1+\e} +C_6 F^{1+\e} |\phi|^{1/2}|\nabla \log u|_{\tilde g(t)}\\
&\leq \frac{n}2 \tilde{\mathcal{R}} G+L^{-2}C_9 \Lambda^{2+2\e}+C_{10}L^2
\end{split}
\end{equation}
and $G=0$ on $\partial B_{\tilde g_0}(x_1,1)\times [0,\tilde t_1]$. By maximum principle, we control $G$ by its kernel representation:
\begin{equation}
\begin{split}
G(x_1,\tilde t_1)&\leq \int_\Omega K_\Omega(x_1,\tilde t_1;y,0)\, G(y,0)\,d\mathrm{vol}_{y,\tilde g(0)}\\
&\quad +C_{11}\left(L^{-2} \Lambda^{2+2\e}+L^2\right)\int^{\tilde t_1}_0\int_\Omega K_\Omega(x_1,\tilde t_1;y,s)\,d\mathrm{vol}_{y,\tilde g(s)}  ds\\
&=\mathbf{IV}+\mathbf{V}
\end{split}
\end{equation}
where $\Omega=B_{\tilde g_0}(x_1,1)$. 

By Proposition~\ref{prop:heatkernelestimate},
\begin{equation}
\begin{split}
\int^{t}_0\int_\Omega K_\Omega(x,t;y,s)\,d\mathrm{vol}_{y,\tilde g(s)}  ds\leq  \int^{t}_0 C_{12} ds=C_{12}t
\end{split}
\end{equation}
so that 
\begin{equation}
\mathbf{V}\leq C_{13}\left(L^{-2} \Lambda^{2+2\e}+L^2\right)\tilde t_1.
\end{equation}

For $\mathbf{IV}$ since $\tilde\ell(0)=0$,
\begin{equation}
\begin{split}
\mathbf{IV}&\leq  \int_\Omega K_\Omega(x_1,\tilde t_1;y,0)\, \left(1+ Lu^\gamma\right)(y,0)\,d\mathrm{vol}_{y,\tilde g(0)}\leq C_{14}L.
\end{split}
\end{equation}

Together with $\tilde\ell (x_1,\tilde t_1)=\frac14 \Lambda\leq \a \tilde t_1^{-1}$, we obtain
\begin{equation}
\begin{split}
\left( \frac14 e^{-4\a K} \Lambda \right)^{1+\e}&\leq C_{14}L+4\a \Lambda^{-1}C_{13}\left(L^{-2} \Lambda^{2+2\e}+L^2\right).
\end{split}
\end{equation}

Therefore, if we choose $L=\Lambda^{(1+\e)/2}$, then 
\begin{equation}
\begin{split}
\left( \frac14 e^{-4\a K} \Lambda \right)^{1+\e}&\leq C_{14}\Lambda^{(1+\e)/2}+8\a C_{13}\Lambda^{\e}
\end{split}
\end{equation}
which is impossible if we choose $\Lambda(n,\a,\b,K)$ large enough to conflict with the above inequality. That said, such $t_1$ does not exist and hence, we have shown that for such $\Lambda$, we have 
\begin{equation}
\ell(x,t)\leq \Lambda \rho(x)^{-2}
\end{equation}
for all $(x,t)\in B_{g_0}(x_0,1)\times [0,T]$. The result follows by taking $x=x_0$.
\end{proof}

A simple modification of the proof shows that indeed $\mathcal{R}(g(t))$ lower bound is preserved along uniformly conformal Yamabe flow. 
\begin{lma}\label{lma:scalar-lower}
Suppose $g_0$ is a metric on $M$ and $x_0\in M$ is a point such that 
\begin{enumerate}
\item $B_{g_0}(x_0,r)\Subset M$;
\item $\Ric(g_0)\geq -(n-1)r^{-2}$.
\end{enumerate}
Let $g(t)$ be a smooth solution to the Yamabe flow on $M\times [0,Tr^2]$ (not necessarily complete) with $g(0)=g_0$ such that $$\a^{-1}g_0\leq g(t)\leq \a g_0$$
on $M\times (0,Tr^2]$ for some $\a>1$ and $r>0$. Then there is $\Lambda_1(n,\a)>0$ such that for all $t\in (0,Tr^2]$, $\mathcal{R}(g(x_0,t))\geq -\Lambda_1r^{-2}$.
\end{lma}
\begin{proof} 
By rescaling, we might assume $r=1$. This follows from a simpler modification of the proof of Lemma~\ref{lma:localMP}. We consider the test function
\begin{equation}
F=\phi\left(d_{ g_0}(x,x_0)\right)\cdot \varphi+L_0u^\gamma 
\end{equation}
where $\varphi=\mathcal{R}_-$, $L_0$ is a large constant, $g(t)=u^\frac4{n-2}g_0$ and $\phi$ is smooth non-increasing function $\phi:[0,+\infty)\to[0,1]$ such that $\phi\equiv 1$ on $[0,\frac14]$, vanishes outside $[0,1]$ and satisfies $|\phi'|^2\leq 10^3\phi,\phi''\geq -10^3\phi$. Then we have
\begin{equation}
\begin{split}
&\quad \yheat F\\
&\leq -(n-1)\Delta_{g(t)} \phi \cdot \varphi-2\langle \nabla \phi,\nabla\varphi\rangle_{g(t)}-\phi\varphi^2\\
&\quad+\frac{\gamma(n-2)}{4} u^{\gamma-\frac4{n-2}}L_0\varphi-L_0 u^{\gamma}|\nabla\log u|^2_{g(t)}
\end{split}
\end{equation}
in the sense of barrier. Here we have used the evolution equation of $u$, see \eqref{eqn:evo-u}. We might assume the function to be smooth when applying maximum principle. 

We now simplify the evolution inequality. We will use $C_i$ to denote any constant depending only on $n,\a$. Using the same derivation of \eqref{eqn:lap-compar} and choice of $\phi$, 
\begin{equation}
\begin{split}
-(n-1)\Delta_{g(t)}\phi \leq C_1\left( 1+\phi^{1/2}|\nabla\log u|_{g(t)}\right)
\end{split}
\end{equation}

On the other hand at its maximum point, $\nabla F=0$ so that $$\phi\nabla \varphi+\varphi \nabla\phi +L_0\gamma u^{\gamma-1}\nabla u=0$$ 
and thus,
\begin{equation}
\begin{split}
-2\langle \nabla\phi ,\nabla\varphi \rangle 
&=2\phi^{-1}\langle \nabla\phi , \varphi \nabla\phi+L_0\gamma u^{\gamma-1}\nabla u \rangle \\
&\leq C_2\varphi +C_2L_0  \phi^{-1/2}|\nabla\log u|_{g(t)}.
\end{split}
\end{equation}

Substituting it back to the evolution equation of $F$ yields 
\begin{equation}
\begin{split}
&\quad \yheat F\\
&\leq  C_1\varphi \left( 1+\phi^{1/2}|\nabla\log u|\right) +C_2\varphi +C_2L_0  \phi^{-1/2}|\nabla\log u|\\
&\quad -\phi\varphi^2+C_3L_0 \varphi-L_0 u^\gamma |\nabla \log u|^2\\
&\leq C_4 L_0\varphi+C_4 \phi^{-1}L_0+C_4L_0^{-1}\varphi^2\phi-\phi\varphi^2
\end{split}
\end{equation}
at its maximum. Hence if we choose $L_0=2C_4$, then at its interior maximum point inside the support of $\phi$, we have 
\begin{equation}
\begin{split}
(\phi\varphi)^2\leq C_5\varphi\phi+C_5 \leq \frac12 (\phi\varphi)^2+C_6.
\end{split}
\end{equation} 

That said, for all $(x,t)\in B_{g_0}(x_0,1)\times [0,T]$, 
\begin{equation}
F(x,t)\leq \sqrt{2C_6}+L_0:=C_7.
\end{equation}
Evaluating at $x_0$ gives us the result.
\end{proof}

\medskip

More generally,  we might  localize its derivation using kernel representation.
\begin{lma}\label{lma:localMP-nonzero}
Suppose $(M,g_0,x_0)$ is a pointed manifold and $g(t),t\in [0,Tr^2]$ is a smooth solution to the Yamabe flow (not necessarily complete) as in Lemma~\ref{lma:localMP}. If $\ell(x,t)$ is a non-negative continuous function on $M\times [0,Tr^2]$ such that $\ell(x,t)\leq \a t^{-1}$ on $M\times (0,Tr^2]$ and satisfies 
\begin{equation}
\yheat \ell\leq \b\mathcal{R}\ell 
\end{equation}
in the barrier sense, for some $0\leq \b<\frac{n}2$, then there is $\hat\Lambda(n,\a,\b),L(n,\b)>0$ such that for all $t\in [0,Tr^2]$,
\begin{equation}
\ell(x_0,t)\leq \hat\Lambda r^{-2}+e^{L t}\int_{\Omega} K_\Omega(x,t;y,0) \,\ell(y,0)\,d\mathrm{vol}_{y,g(0)}
\end{equation}
where $\Omega=B_{g_0}(x_0,r)$.
\end{lma}
\begin{proof}
By rescaling, we assume $r=1$. We let $\phi$ be a cutoff function such that $\phi\equiv 1$ on $B_{g_0}(x_0,\frac12)$ and vanishes outside $B_{g_0}(x_0,1)$. Consider the function $u=(\ell-e^{Lt}\hat\ell)_+$ where
\begin{equation}
\hat\ell(x,t):=\int_\Omega  K_\Omega(x,t;y,0) \,(\phi\cdot \ell)(y,0)\,d\mathrm{vol}_{y,g(0)}.
\end{equation}

Whenever $u>0$, it satisfies 
\begin{equation}
\begin{split}
\yheat u&\leq   \b \mathcal{R} u-\left[\left(\frac{n}2-\b\right)\mathcal{R}  +L\right] e^{Lt}\hat \ell\\
&\leq  \b \mathcal{R} u
\end{split}
\end{equation}
in the sense of barrier, provided that we choose $L= n(n-1)\left(\frac{n}2-\b\right)$. The evolution inequality is also in distributional sense, \cite{MantegazzaMascellaniUraltsev2014}. Since $\hat \ell\geq 0$ and $u(0)=0$ on $B_{g_0}(x_0,\frac12)$, we might now apply Lemma~\ref{lma:localMP} to $u$ so that $
u(x_0,t)\leq 4\Lambda $ for all $t\in [0,T]$, where $\Lambda=\Lambda(\a,\b,0,n)$ is the constant from Lemma~\ref{lma:localMP}. This completes the proof.
\end{proof}
\medskip

As an application of the localized maximum principle, we have local persistence of Ricci lower bound.
\begin{prop}\label{prop:Ric-preserved}
Suppose $(M^n,g_0,x_0)$ is a pointed manifold with $n\geq 5$, and $g(t),t\in [0,Tr^2]$ is a smooth solution to the Yamabe flow (not necessarily) complete as in Lemma~\ref{lma:localMP}. If in addition $g_0$ is locally conformally flat,  there is $\Lambda(n,\a)>0$ such that for all $t\in [0,Tr^2]$,
\begin{equation}
\Ric(x_0,t)\geq -\Lambda r^{-2}.
\end{equation}
\end{prop}
\begin{proof}
Since $g(t)$ is conformal to $g_0$, $g(t)$ is locally conformally flat for all $t\in [0,Tr^2]$. The result follows from Lemma~\ref{lma:evo-Ric-new}, Lemma~\ref{lma:localMP} and $n\geq 5$.
\end{proof}

\medskip

Another consequence is a psuedo-locality type estimate when the curvature is initially pinched from flat manifold. 
\begin{prop}\label{prop:pseudo-YF}
Suppose $(M^n,g_0,x_0)$ is a pointed locally conformally flat manifold with $n\geq 5$, and $g(t),t\in [0,Tr^2]$ is a smooth solution to Yamabe flow (not necessarily complete) as in Lemma~\ref{lma:localMP}. If in addition $x_0$ satisfies
\begin{equation}
\int^{r^2}_0 s \fint_{B_{g_0}(x_0,s)} |\mathcal{R}_{g_0}|\,d\mathrm{vol}_{g_0}\,ds\leq \Lambda_0
\end{equation}
for some $\Lambda_0>0$, then there is $L(n,\a),\Lambda_1(n,\a)>0$ such that for all $t\in (0,Tr^2]$,
\begin{equation}
\mathcal{R}(x_0,t)\leq \Lambda_1 r^{-2}+L\Lambda_0 t^{-1}.
\end{equation}
\end{prop}
\begin{proof}
By rescaling, we will assume $r=1$. By Lemma~\ref{lma:evo-R}, $\ell(x,t):=(\mathcal{R}_{g(t)}(x))_+$ satisfies 
\begin{equation}
\begin{split}
\yheat \ell=\mathcal{R}^2\leq \mathcal{R}\ell
\end{split}
\end{equation}
in the sense of barrier, whenever $\ell>0$. This this fulfils the assumption in Lemma~\ref{lma:localMP-nonzero} with $\b=1$. In particular, 
\begin{equation}\label{eqn:localized-scalar-kernel}
\ell(x_0,t)\leq \hat \Lambda +e^{\hat Lt} \int_{\Omega} K_\Omega(x,t;y,0) \mathcal{R}_+(y)\,d\mathrm{vol}_{y,g_0}
\end{equation}
where $\Omega=B_{g_0}(x_0,1)$.

We now estimate the kernel representation. We use $C_i$ to denote any constant depending only on $n,\a$. We also denote 
\begin{equation}
V(r):=\mathrm{Vol}_{g_0}\left(B_{g_0}(x_0,r)\right)\cdot k(x_0,r)=\int_{B_{g_0}(x_0,r)} \mathcal{R}_+(y)\,d\mathrm{vol}_{g_0}
\end{equation}
for notation convenience. By  Proposition~\ref{prop:heatkernelestimate} and co-area formula, 
\begin{equation}
\begin{split}
&\quad \int_{\Omega} K_\Omega(x,t;y,0) \mathcal{R}_+(y)\,d\mathrm{vol}_{y,g_0}\\
&\leq \int_{B_{g_0}(x_0,1)} \frac{C_1}{\mathrm{Vol}_{g_0}\left(B_{g_0}(x_0,\sqrt{t}) \right)}\cdot \exp\left(-\frac{d_{g_0}(x,y)^2}{C_1t} \right)\mathcal{R}_+(y)\,d\mathrm{vol}_{y,g_0}\\
&=\int^1_0 \frac{C_1}{\mathrm{Vol}_{g_0}\left(B_{g_0}(x_0,\sqrt{t}) \right)}\cdot \exp\left(-\frac{r^2}{C_1t} \right) V'(r)\, dr\\
&=\frac{C_1}{\mathrm{Vol}_{g_0}\left(B_{g_0}(x_0,\sqrt{t}) \right)}\cdot \exp\left(-\frac{1}{C_1t} \right) V(1)\\
&\quad +\int^1_0 \frac{2r V(r)}{t\cdot \mathrm{Vol}_{g_0}\left(B_{g_0}(x_0,\sqrt{t}) \right)}\cdot \exp\left(-\frac{r^2}{C_0t} \right) dr\\
&=\mathbf{I}+\mathbf{II}.
\end{split}
\end{equation}

By volume comparison using $g_0$ and $\Ric(g_0)\geq -(n-1)$, for $t\in (0,1]$,
\begin{equation}\label{eqn:I-scalar-local-1}
\begin{split}
\mathbf{I}&\leq \frac{\mathrm{Vol}_{g_0}\left(B_{g_0}(x_0,1) \right)}{\mathrm{Vol}_{g_0}\left(B_{g_0}(x_0,\sqrt{t}) \right)}\cdot \exp\left(-\frac{1}{C_1t} \right) C_1\cdot k(x_0,1)\\
&\leq \frac{C_2}{t^{n/2}}\exp\left(-\frac{1}{C_1t} \right)  \Lambda_0\leq C_3\Lambda_0 t^{-1}
\end{split}
\end{equation}
where we used \cite[Lemma 3.1]{ChanLee2023}.

Similarly,
\begin{equation}
\begin{split}
\mathbf{II}&= \left(\int^1_{\sqrt{t}}+\int^{\sqrt{t}}_0\right)\frac{2r}{t}\cdot \frac{ \mathrm{Vol}_{g_0}\left(B_{g_0}(x_0,r) \right)}{ \mathrm{Vol}_{g_0}\left(B_{g_0}(x_0,\sqrt{t}) \right)}\cdot \exp\left(-\frac{r^2}{C_0t} \right)\cdot k(x_0,r) dr\\
&=\mathbf{III}+\mathbf{IV}.
\end{split}
\end{equation}

The straightforward set inclusion yields 
\begin{equation}\label{eqn:I-scalar-local-2}
\begin{split}
\mathbf{IV}\leq \frac{2}{t} \int^{\sqrt{t}}_0 r\, k(x_0,r) \, dr\leq 2\Lambda_0 t^{-1}
\end{split}
\end{equation}
while we can apply volume comparison to deduce 
\begin{equation}\label{eqn:I-scalar-local-3}
\begin{split}
\mathbf{III}&\leq  \frac{C_1}{t}\int^1_{\sqrt{t}} \left(\frac{r}{\sqrt{t}}\right)^{n}\cdot \exp\left(-\frac{r^2}{C_0t} \right)\cdot  r\,k(x_0,r) dr\leq  C_2\Lambda_0 t^{-1}.
\end{split}
\end{equation}

The assertion follows from combining \eqref{eqn:localized-scalar-kernel} with \eqref{eqn:I-scalar-local-1}, \eqref{eqn:I-scalar-local-2} and \eqref{eqn:I-scalar-local-3}.
\end{proof}

\section{Existence of immortal Yamabe flow}

In this section, we will construct long-time solution to locally conformally flat Yamabe flow from metrics. When the initial metric $g_0$ is of bounded curvature, this had been considered by Ma \cite{Ma2016}, under a slightly weaker asymptotic assumption. 

\subsection{Barrier function from Poisson equation}

The method by Ma is based on constructing barrier function from solution to Poission equation. We need part of the argument in \cite{Ma2016}.

\begin{lma}\label{lma:poisson}
Suppose $(M^n,g_0)$ is a complete non-compact manifold such that $\Ric\geq0$ and 
\begin{equation}\label{eqn:ass-kxr}
\int^\infty_0 r k(x,r)\,dr \leq \Lambda_0
\end{equation}
for all $x\in M$ and $\Lambda_0>0$, then there exists a bounded non-positive smooth function $w$ on $M$ such that $\Delta_{g_0}w=\frac{n-2}{4(n-1)}\mathcal{R}(g_0)$. 
\end{lma}
\begin{proof}
The existence of $ w$ where $\Delta \hat u=\mathcal{R}$ follows from  \cite[Corollay  2.1]{NiShiTam2001} and translation. 
\end{proof}

\medskip

In particular, the existence of $u$ from Lemma~\ref{lma:poisson} allow us to construct Yamabe flow which is uniformly conformal.
\begin{prop}\label{prop:immortal-sol-1}
Suppose $(M,g_0)$ is a complete non-compact manifold as in Lemma~\ref{lma:poisson}. Then there is a immortal solution $g(t),t\in [0,+\infty)$ to the Yamabe flow starting from $g_0$ such that $$e^\frac{4w}{n-2}g_0\leq g(t)\leq  g_0$$ on $M\times [0,+\infty)$ where $w(x)$ is the function obtained from Lemma~\ref{lma:poisson}. Furthermore, $\mathcal{R}(g(t))\geq 0$ for all $t\in [0,+\infty)$.
\end{prop}
\begin{proof}
This follows from a similar argument of \cite[Theorem 5]{Ma2016}, see also \cite{ChenZhu2002}. For the sake of convenience, we include the proof for readers' convenience.

We consider the dirichlet problem:
\begin{equation}\label{Dirichletproblemofu}
\begin{cases}
\displaystyle   \partial_tu^N = (n-1)N\left[\Delta_{g_0}u-\frac{n-2}{4(n-1)}\mathcal{R}_{g_0}u\right],\text{ if }(x,t)\in\Omega\times(0,\infty);\\[2mm]
    u(x,t)>0,\text{ if }(x,t)\in\Omega\times(0,\infty);\\[2mm]
    u(x,0)=1,\text{ if }x\in\Omega;\\[2mm]
    u(x,t)=1,\text{ if }(x,t)\in\partial\Omega\times(0,T).
\end{cases}
\end{equation}
where $N=\frac{n+2}{n-2}$, on pre-compact domain $\Omega$ in $M$ with smooth $\partial\Omega$. This is equivalent to the Yamabe flow on $\Omega$ in the sense that $g(t):=u^\frac4{n-2}g_0$ is a solution to Yamabe flow. 

We take a compact exhaustion $\{\Omega_k\}_{k=1}^\infty$ of $M$ with smooth boundary. By \cite[Proposition 2.2]{ChenZhu2002}, each $\Omega_k$ admits an unique classical solution $u_k$ on $\Omega_k\times [0,+\infty)$. Furthermore since $\mathcal{R}_{g_0}\geq 0$, maximum principle implies $u_k(x,t)\leq 1$ on $\Omega_k\times [0,+\infty)$. 

We now show the lower bound. Fix $T>0$ and let $w$ be the function such that $\frac{n-2}{4(n-1)} \mathcal{R}_{g_0}=\Delta_{g_0}w$ on $M$ obtained from Lemma~\ref{lma:poisson}. Then the function $v=\log u_k$ satisfies
\begin{equation}
\partial_t (v-w) \geq (n-1)e^{-(N-1)v} \cdot  \Delta_{g_0}(v-w) 
\end{equation}
where $v=0$ on $\Omega\times \{0\}\bigcup \partial\Omega\times (0,T)$. The function $v_\e:=v-w+\e t$ therefore satisfies 
\begin{equation}
\partial_t v_\e \geq \e +(n-1) e^{-(N-1)v}\Delta_{g_0}v_\e
\end{equation}
and hence minimum principle implies that its minimum is not attained at interior point. Since $w$ is non-positive, by letting $\e\to0$ we see that $u_k\geq w$ on $\Omega_k\times [0,T]$ for all $k$ and $T>0$. Hence, we have shown that for all $k\in\mathbb{N}$ and $t\in [0,+\infty)$, 
\begin{equation}
1\geq u_k(x,t)\geq w.
\end{equation}

Now the existence of immortal solution on $M$  uniformly conformal to initial data, follows from sub-sequential limit of $u_k$ using local parabolic Schauder estimate. 
\end{proof}

\begin{rem}
Indeed by \cite[Theorem 1.1]{NiShiTam2001}, in order to ensure existence using Dirichlet approximation, it is sufficient to requires \eqref{eqn:ass-kxr} to hold for some $x_0\in M$ instead of all $x\in M$. The non-negativity of scalar curvature follows from applying Lemma~\ref{lma:scalar-lower} with $r\to+\infty$.
\end{rem}

\subsection{Curvature estimate along Yamabe flow}
In this section, we want to show that the solution obtained from Proposition~\ref{prop:immortal-sol-1} has uniform curvature estimate if the initial metric $g_0$ satisfies the assumptions in Theorem~\ref{thm:main}. 

\begin{prop}\label{prop:cur-est-k}
For $n\geq 5$. Suppose $g(t),t\in [0,+\infty)$ is a complete solution to the Yamabe flow such that $
\a^{-1}g_0\leq g(t)\leq \a g_0$
for some $\a>1$ and for all $t\geq 0$. There exists $\e_0(n,\a)>0$ such that if the initial metric $g_0$ is locally conformally flat and satisfies 
\begin{enumerate}
\item[(i)] $\Ric(g_0)\geq 0$;
\item[(ii)] $ \int^\infty_0 r \fint_{B_{g_0}(x,r)}\mathcal{R}(g_0)\, d\mathrm{vol}_{g_0}\,  dr\leq \e_0$, for all $x\in M$, 
\end{enumerate}
then the Yamabe flow $g(t)$ satisfies 
$\Ric(g(t))\geq 0$ and $|\Rm(g(t))|\leq 2t^{-1}$ on $M\times (0,+\infty)$.
\end{prop}

Before we prove Proposition~\ref{prop:cur-est-k}, we need a Lemma allowing us to bootstrap from uniform conformal to metric with bounded curvature. 
\begin{lma}\label{lma:local-doubling}
Let $g(t)$ be a solution to locally conformally flat Yamabe flow with initial metric $g_0$ and $x_0\in M$ be such that $B_{g_0}(x_0,r)\Subset M$, $|\Rm(g_0)|\leq r^{-2}$ on  $B_{g_0}(x_0,r)$ and 
\begin{equation}
\a^{-1}g_0\leq g(t)\leq \a g_0
\end{equation}
for $t\in [0,T]$ and for some $\a>1$. Then there exists $\hat T(n,\a), \hat C(n,\a)>0$ such that $|\Rm(g(x_0,t))|\leq \hat  Cr^{-2}$ for all $t\in [0,T\wedge \hat Tr^2]$.
\end{lma} 
\begin{proof}
By scaling, we might assume $r=1$. This then follows from the proof of \cite[Theorem 2.4]{Cheng2025} which is based on argument of Shi in \cite{Shi1989}. We remark that the argument of Shi is purely local. 
\end{proof}

Now we are ready to prove Proposition~\ref{prop:cur-est-k} using idea from Hochard \cite{HochardPhd}.
\begin{proof}[Proof of Proposition~\ref{prop:cur-est-k}]

Fix $x_0\in M$. We might assume $\a>10$ for convenience. We might also assume 
\begin{equation}
\a^{-1}g(t)\leq g(s)\leq \a g(t)
\end{equation}
for all $t,s\in [0,+\infty)$. Since the assumptions are scaling invariant, we first claim that $|\Rm(g(t))|\leq 2t^{-1}$ on $B_{g_0}(x_0,1)\times (0,T]$ for some $T(n,\a)>0$. 

Since $B_{g_0}(x_0,2)\Subset M$, we might find $t_0>0$ such that $|\Rm(g(t))|\leq 2t^{-1}$ on $B_{g_0}(x_0,r_0)\times (0,t_0]$ where $r_0=2$. We fix $L_0=\e_0^{-1/2}$. Applying Proposition~\ref{prop:Ric-preserved} on $B_{g_0}(x,L_0\sqrt{t_0})\times [0,t_0]$ where $x\in B_{g_0}(x_0,r_0-L_0\sqrt{t_0})$, we see that 
\begin{equation}
\Ric(g(t))\geq -L_0^{-2}\Lambda(n,\a)t_0^{-1}
\end{equation}
on $B_{g_0}(x_0,r_0-L_0\sqrt{t_0})\times [0,t_0]$. Now we apply Proposition~\ref{prop:pseudo-YF} on $B_{g_0}(x,L_0\sqrt{t_0})\times [0,t_0]$ where $x\in B_{g_0}(x_0,r_0-2L_0\sqrt{t_0})$ to see that 
\begin{equation}
\mathcal{R}(g(t))\leq \Lambda_1(n,\a)L_0^{-2}t_0^{-1}+L(n,\a) \e_0 t^{-1}
\end{equation}
for all $(x,t)\in B_{g_0}(x_0,r_0-2L_0\sqrt{t_0})\times (0,t_0]$. Since $g(t)$ is locally conformally flat, we conclude that 
\begin{equation}
\begin{split}
|\Rm(g(t_0))|&\leq C_n|\Ric(g(t_0))|\\
&\leq C_n \left(\Lambda  + \Lambda_1 +L  \right)\e_0  t_0^{-1}:=C_1(n,\a)\e_0 t_0^{-1}
\end{split}
\end{equation}
on $B_{g_0}(x_0,r_0-2L_0\sqrt{t_0})$.

We apply Lemma~\ref{lma:local-doubling} to the translated Yamabe flow with $r=\sqrt{C_1^{-1}\e_0^{-1}t_0}$ so that 
\begin{equation}
|\Rm(g(t))|\leq \hat C \cdot C_1\e_0t_0^{-1}
\end{equation}
on $B_{g_0}(x_0,r_0-\tilde L_0\sqrt{t_0})\times [t_0,(1+\tilde T)t_0]$ where $$\tilde L_0=2L_0+\sqrt{C_1^{-1}\e_0^{-1}}\quad\text{and}\quad \tilde T=\hat T\cdot C_1^{-1}\e_0^{-1}.$$

We choose $\e_0$ small enough so that $\tilde T>1$ and $\hat C \cdot C_1\e_0 \leq 1$. Hence, 
\begin{equation}
|\Rm(g(t))|\leq 2t ^{-1}
\end{equation}
on $B_{g_0}(x_0,r_0-\tilde L_0\sqrt{t_0})\times [t_0,2t_0]$. We now define $t_i$ and $r_i$ inductively:  $t_{i+1}=2t_i$ and $r_{i+1}=r_i-\tilde L_0\sqrt{t_i}$ for $i\geq 0$. 

By repeating the above argument inductively, we see that if $r_i>0$, then we have $|\Rm(g(t))|\leq 2t^{-1}$ on $(0,t_i]$. We consider $r_N>1\geq r_{N+1}$. We see that $t_N\geq  T(n,\a)$ and thus prove the claim. A rescaling argument show that $|\Rm|\leq 2t^{-1}$ on $M\times (0,+\infty)$. The non-negativity of $\Ric$ follows from Proposition~\ref{prop:Ric-preserved} with $r\to+\infty$.
\end{proof}

Now we are ready to prove Theorem~\ref{thm:main}.
\begin{proof}[Proof of Theorem~\ref{thm:main}]

By Proposition~\ref{prop:immortal-sol-1} and Proposition~\ref{prop:cur-est-k}, we obtain a complete solution to the Yamabe flow $g(t)$ with $0\leq \Ric(g(t))\leq 2t^{-1}$ and 
$$\a^{-1}g_0\leq g(t)\leq g_0$$
for all $t>0$. Now we are able to use the argument of Ma \cite{Ma2016} to draw the flatness using Chow's harnack inequality \cite{Chow1992}, which is valid thanks to bounded curvature for $t>0$. We include the proof for readers' convenience. 

By Harnack inequality, $t\mathcal{R}(x,t)$ is non-decreasing for $t>0$. Hence, for $t\geq1$ and $\tau\in[\sqrt{t},t]$, we have $\tau\mathcal{R}(\tau) \geq\sqrt{t}\mathcal{R}(\sqrt{t})$. The equation of Yamabe flow and metric equivalence imply
    \begin{equation}
        \int_{0}^{t}\mathcal{R}(x,\tau)d\tau=-\frac{4}{n-2}\log u(x,t) \leq C(n,\a)
    \end{equation}
    for all $x\in M$. Using monotonicity, we get
    \begin{equation}
    \begin{split}
  \frac12 \log t\cdot \sqrt{t}\mathcal{R}(x,\sqrt{t}) \leq \int^{t}_{\sqrt{t}}\mathcal{R}(x,\tau)d\tau\leq C(n,\a)
    \end{split}
    \end{equation}
    so that $\lim_{t\to +\infty}t\mathcal{R}(x,t)=0$. By monotonicity, this implies $\mathcal{R}(x,t)=0$ for all $t>0$. Since the flow is smooth up to $t=0$, this proves the Theorem by letting $t\to 0$.
\end{proof}

\end{document}